\documentclass[12pt]{amsart}   	
\usepackage{fullpage}                		

\usepackage{pdfpages}
\usepackage{amsfonts}
\usepackage{amscd}
\usepackage{amsmath}
\usepackage{amsthm}
\usepackage{upgreek}
\usepackage{mathdots}
\usepackage{multirow}
\usepackage{array}
\usepackage{tablefootnote}
\usepackage{wasysym}					
\usepackage{mathtools}
\usepackage{float}
\usepackage{placeins}
\usepackage{stmaryrd}
\usepackage{algorithm2e}
\usepackage{caption}
\usepackage{subcaption}
\usepackage{changepage}
\usepackage{tikz}
\usepackage{tikz-cd}
\usepackage{listings}
\usepackage[all,cmtip]{xy}
\usetikzlibrary{matrix,arrows}
\usepackage{amssymb,mathrsfs}
\usepackage[shortlabels]{enumitem}

\definecolor{hot}{RGB}{65,105,225}
\usepackage[pagebackref=true,colorlinks=true, linkcolor=hot ,  citecolor=hot, urlcolor=hot]{hyperref}

\theoremstyle{plain}
\newtheorem{theorem}{Theorem}[section]
\newtheorem{thm}[theorem]{Theorem}
\newtheorem{lemma}[theorem]{Lemma}
\newtheorem{prop}[theorem]{Proposition}

\newtheorem{con}[theorem]{Conjecture}
\newtheorem{conj}[theorem]{Conjecture}

\theoremstyle{definition}
\newtheorem{rmk}[theorem]{Remark}

\newtheorem{defi}[theorem]{Definition}
\newtheorem{setup}[theorem]{Setup}
\newtheorem{subs}[theorem]{}

\newcommand{\ol}[1]{\overline{#1}}

\newcommand{\bbc}{\mathbb{C}}

\makeatletter

\newcommand{\xdashrightarrow}[2][]{\ext@arrow 0359\rightarrowfill@@{#1}{#2}}
\def\rightarrowfill@@{\arrowfill@@\relax\relbar\rightarrow}
\def\arrowfill@@#1#2#3#4{%
  $\m@th\thickmuskip0mu\medmuskip\thickmuskip\thinmuskip\thickmuskip
   \relax#4#1
   \xleaders\hbox{$#4#2$}\hfill
   #3$%
}

\newcommand\restr[2]{{
  \left.\kern-\nulldelimiterspace 
  #1 
  \vphantom{\big|} 
  \right|_{#2} 
  }}

\makeatother
\newcommand{\ord}{\text{ord}}

\def\al{\alpha}
\def\bC{\mathbb{C}}
\def\ra{\rightarrow}
\def\cS{\mathcal{S}}
\def\cP{\mathcal{P}}
\def\cA{\mathcal{A}}
\def\sD{\mathscr{D}}

\def\bP{\mathbb{P}}
\def\bR{\mathbb{R}}
\def\be{\begin{equation}}
\def\ee{\end{equation}}
\def\cO{\mathcal{O}}
\def\bN{\mathbb{N}}
\def\bZ{\mathbb{Z}}
\def\bQ{\mathbb{Q}}
\def\cL{\mathcal{L}}
\def\xa{\xrightarrow}
\def\bbY{\bar{Y}}
\def\bbm{\bar{\mu}}

\title[Monodromy Conjecture for log generic polynomials]{Monodromy Conjecture for log generic polynomials}

\author{Nero Budur}
\address{KU Leuven, Celestijnenlaan 200B, B-3001 Leuven, Belgium, and BCAM, Mazarredo 14, 48009 Bilbao, Spain.}
\email{nero.budur@kuleuven.be}

\author{Robin van der Veer} 
\address{KU Leuven, Celestijnenlaan 200B, B-3001 Leuven, Belgium.}
\email{robin.vanderveer@kuleuven.be}

\keywords{Monodromy conjecture; motivic zeta function; Milnor fibration; Bernstein-Sato ideal.}

\subjclass[2020]{14E18, 14F10, 32S40.}

\begin{document}

\begin{abstract} 
A log generic hypersurface in $\bP^n$ with respect to a birational modification of $\bP^n$  is by definition the image of a generic element of a high power of an ample linear series on the modification. A log very-generic hypersurface is defined similarly but restricting to line bundles satisfying a non-resonance condition. Fixing a log resolution of a product $f=f_1\ldots f_p$ of polynomials, we show that the monodromy conjecture, relating the motivic zeta function with the complex monodromy, holds for the tuple $(f_1,\ldots,f_p,g)$ and for the product $fg$, if $g$ is  log generic. We also show that the stronger version of the monodromy conjecture, relating the motivic zeta function with the Bernstein-Sato ideal, holds for the tuple $(f_1,\ldots,f_p,g)$ and for the product $fg$, if $g$ is log very-generic. Even the case $f=1$ is intricate, the proof depending on nontrivial properties of Bernstein-Sato ideals, and it singles out the class of log (very-) generic hypersurfaces as an interesting class of singularities on its own. \end{abstract}

\maketitle


\section{Introduction}

Let $F=(f_1,\ldots, f_p)$ be a tuple of polynomials $f_i\in\bC[x_1,\ldots,x_n]$. Let $f=\prod_{i=1}^pf_i.$  The {\it topological zeta function} of $F$ $$Z_F^{top}(s_1,\ldots,s_p)$$ is a rational function, cf. Definition \ref{defTZ}. We denote the polar locus of this rational function, that is, the support of the divisor of poles, by $\mathcal{P}(Z_F^{top})$.

On the other hand, one has the {\it monodromy support} of $F$ $$\mathcal S _F\subset (\bC^*)^p,$$
cf. Definition \ref{defMS}. If $p=1$, this is the set of all eigenvalues of the monodromy on the cohomology of the Milnor fibers of $f$. Let $Exp:\bbc^p\to (\bbc^*)^p$ be the map $\alpha\mapsto \exp(2\pi i \alpha)$ coordinate-wise.

\begin{con}[Monodromy Conjecture]
Let $F$ be a tuple of polynomials in $\bC[x_1,\ldots,x_n]$. Then
$$Exp(\mathcal{P}(Z_F^{top}))\subset \mathcal{S}_F.$$
\end{con}

A stronger conjecture involves the {\it Bernstein-Sato ideal} $B_F$, the ideal of generated by $b\in\bC[s_1,\ldots,s_p]$ satisfying
$$
b\prod_{i=1}^pf_i^{s_i}=P\prod_{i=1}^pf_i^{s_i+1}
$$
for some $P\in\sD[s_1,\ldots,s_{p}]$, where $\sD$ is the ring of linear algebraic differential operators on $\bC^n$.
When $p=1$, the monic generator of this ideal is the {\it $b$-function} of $f$. Let $Z(B_F)$ denote the zero locus of $B_F$ in $\bC^p$. It was recently proven in \cite{BVWZ} that $$Exp(Z(B_F))=\cS_F,$$
extending the case $p=1$ due to Malgrange, Kashiwara. 

\begin{conj}[Strong Monodromy Conjecture] \label{conj2}
Let $F$ be a tuple of polynomials in $\bC[x_1,\ldots,x_n]$. Then
$$
\cP(Z_F^{top})\subset Z(B_F).
$$
\end{conj}

For $p=1$, the  conjectures are the analog due to Denef-Loeser \cite{DL} of a classical conjecture for $p$-adic local zeta functions of Igusa \cite{I}. 

Among the known cases with $p=1$ of the stronger conjecture are:  plane curves \cite{L88}, tame hyperplane arrangements \cite{W}. Among the known cases with $p=1$ of the weaker version are: hyperplane arrangements \cite{BMT}, non-degenerate surfaces \cite{LV} and non-degenerate threefolds \cite{Le3}. See the survey \cite{Me} for more cases.

For $p> 1$, the Monodromy Conjecture was posed by Loeser, cf.  \cite{Ni}, see also \cite{L-pl}. It is known for tuples of plane curves \cite{Ni}, and of hyperplane arrangements \cite{B-ls}.  

For $p>1$, the Strong Monodromy Conjecture was posed in \cite{B-ls}. It is known for tuples factorizing a tame hyperplane arrangement \cite{Ba}, and for tuples of linear polynomials \cite{Wu}.

In this note we address both conjectures in presence of two notions of genericity with respect to birational modifications. 
\begin{setup}\label{eqDi}
We fix a non-zero polynomial $f\in\bC[x_1,\ldots,x_n]$ and maps
$$
\xymatrix{
 Y \ar[d]^\mu  \ar@{^{(}->}[r]& \bar{Y} \ar[d]^{\bar{\mu}}\\
 \bC^n \ar@{^{(}->}[r] & \bP^n
}
$$
such that:

\begin{itemize}
\item the bottom map is the inclusion of the complement of the hyperplane at infinity, $Y=\bar{\mu}^{-1}(\bC^n)$, and $\mu=\bar{\mu}|_Y$;
\item $\bar{\mu}$ is a composition of blowing ups of smooth closed subvarieties;
\item $\bar{\mu}$ is a log resolution of the divisor $\text{div}(f)$ in $\bP^n$ of the rational function $f$, that is, the union of the exceptional locus of $\bar{\mu}$ with $\text{div}(f)$ is a simple normal crossings divisor in $\bar Y$.
\end{itemize}
\end{setup}

Starting with $f$, one can always reach such a setup. Having fixed this set-up, we make two definitions. The first one is:

\begin{defi}\label{defVG}  We say that a statement holds for {\it  log generic polynomials}  in $\bC[x_1,\ldots,x_n]$ if:
 
- for every  ample line bundle $L$ on $\bar{Y}$, and

- for all $k\gg 0$, 

\noindent the statement holds for $g\in\bC[x_1,\ldots,x_n]$, where $g$ is a defining polynomial for the image under $\mu$ of the restriction to $Y$ of a generic member of the finite dimensional space $|L^{\otimes{k}}|$. Here $k\gg 0$ means $k\ge k_0$ for some $k_0$ depending on $L$, $\bar\mu$, and $f$.

\end{defi}


For the second definition we will restrict to very general line bundles, that is, elements of a certain non-empty subcone $\cA^{vg}$ of the integral ample cone of $\bar{Y}$ as in Definition \ref{defAvg}.

\begin{defi}\label{defVGb}
We say that a statement holds for {\it log very-generic polynomials} in $\bC[x_1,\ldots,x_n]$ if:

- for every ample line bundle $L$ on $\bar{Y}$ such that $L\in \cA^{vg}$, and

- for all $k\gg 0$, 

\noindent the statement holds for $g\in\bC[x_1,\ldots,x_n]$, where $g$ is a defining polynomial for the image under $\mu$ of the restriction to $Y$ of a generic member of the finite dimensional space $|L^{\otimes{k}}|$. As above,  $k\gg 0$ means $k\ge k_0$ for some $k_0$ depending on $L$, $\bar\mu$, and $f$.

\end{defi}

\begin{rmk}\label{rmkGg}$\,$

(i) If a statement holds log generically then it holds log very-generically.

(ii) The morphism $\mu$ is also a log resolution of $fg$ for any log generic polynomial $g$, by Bertini Theorem. Moreover $\mu$ is a minimal log resolution for $g$, in the sense that if $\mu$ factors through another log resolution $\mu'$ of $f$, then $\mu'$ is not a log resolution of $g$. 

(iii) Even if $f=1$, log generic polynomials can be highly singular depending on the log resolutions chosen.

(iv) Log generic polynomials $g$ can be obtained from generic elements of symbolic powers of ideals as follows. If 
$$
L \simeq \bar{\mu}^*(\mathcal O_{\bP^n}(d))\otimes_{\mathcal O_{\bar{Y}}} \cO_{\bar{Y}} (-A)
$$
for some positive integer $d$ and some effective divisor $A$ supported on the exceptional locus of $\bar{\mu}$ such that $-A$ is relatively ample,
then
$$
\bar{\mu}_*(L^{\otimes{k}}) \simeq \mathcal{J}^{(k)}(kd)
$$
where $\mathcal J$ is the ideal subsheaf $\mu_*(\cO_{\bar{Y}}(-A))$ of $\cO_{\bP^n}$, and $\mathcal{J}^{(k)}=\mu_*(\cO_{\bar{Y}}(-kA))$ is the $k$-th symbolic power of $\mathcal J$. Then $g$ are generic elements for $k\gg 0$ of the image of the restriction map\be\label{eqIm}\Gamma(\bP^n,\mathcal{J}^{(k)}(kd))\ra \Gamma(\bC^n,\mathcal{J}^{(k)}(kd)).\ee

(v) Log generic polynomials are not necessarily non-degenerate polynomials even in the case $f=1$, although there is an analogy.  Non-degenerate polynomials are generic elements in finite-dimensional vector spaces, generated by monomials, of polynomials with fixed Newton polytope, see \ref{subnondeg}. However, by taking log resolutions of non-monomial ideals, one can generate examples for which the image of the map (\ref{eqIm}) cannot be generated by monomials. 

(vi) The condition imposed on log genericity to obtain log very-genericity is an analog of the non-resonance condition for non-degenerate polynomials of \cite{Lo2}.

We expand on these remarks in Remark  \ref{rmkXw} below in the context of the (Strong) Monodromy Conjecture.


\end{rmk}

To state the results, we keep Setup \ref{eqDi} fixed.

\begin{theorem}\label{thmMain} Let $F=(f_1,\ldots,f_p)$ be a tuple of non-zero polynomials in $\bC[x_1,\ldots,x_n]$, and $f=\prod_{i=1}^pf_i$. Then:
\begin{enumerate}[(a)]
\item the Monodromy Conjecture for $\tilde{F}=(f_1,\ldots,f_{p},g)$ is true, for log generic polynomials $g$;
\item the Strong Monodromy Conjecture for $\tilde{F}=(f_1,\ldots,f_{p},g)$ is true, for log very-generic polynomials $g$.
\end{enumerate}
\end{theorem}

The same holds after taking products:

\begin{theorem}\label{thmMain2} With the same setup:
\begin{enumerate}[(a)]
\item the Monodromy Conjecture for the product $fg$ is true, for log generic polynomials $g$;
\item the Strong Monodromy Conjecture for the product $fg$ is true, for log very-generic polynomials $g$.
\end{enumerate}
\end{theorem}

\begin{rmk}\label{rmkXw} Consider log (very-)generic polynomials in the sense of both theorems. 

(i) The numbers $k_0$ and the genericity condition in the definition of log (very-)generic polynomials can be described geometrically. In \ref{subk} we provide a bound for $k_0$ if $n=2$. In general, see  (\ref{eqG1}) and (\ref{eqG2}) (resp. (\ref{eqCx}) and (\ref{eqSMX})) for a few, but not all, geometric conditions. Equivalently, the set of log (very-)generic polynomials can be defined as the largest set of polynomials for which the proofs of the theorems work. This leads to the next example.

(ii)
If $g(x_1,\ldots,x_n)$ is a homogeneous polynomial with an isolated singularity at the origin, then $g$ is log generic for the setup: $f=1$, $\mu$ is the blowup at the origin $O\in\bC^n$, and $\bar\mu$ is the blowup of $\bP^n$ at $O$. Here $L=\bar\mu^*\cO_{\bP^n}(d)\otimes\cO_{\bar{Y}}(-dW)$, where $W$ is the exceptional divisor and $d$ is the degree of $g$,  $k_0=1$, and $g$ arises from an irreducible element of $|L|$ whose union with $W$ has simple normal crossing singularities. Also, $d>n$ iff $g$ is log very-generic, due to the nef and big condition (\ref{eqCx}) used in the proofs. Theorem \ref{thmMain2}  for this case has been extended more generally to semi-quasihomogeneous hypersurfaces in \cite{BBV}. 

(iii) The polynomial $(x_1+x_2)^2+x_1x_3+x_3^2$ is degenerate (cf. \cite[1.21]{Ko}), log generic in the setup of (ii) since it is homogeneous with an isolated singularity, but not log very-generic since $d=2\le n=3$ by (ii).  Another simple example of Setup \ref{eqDi} providing log generic polynomials which are not log very-generic is given in Remark \ref{rmkLa} (i).

(iv) A general procedure for constructing log (very-)generic polynomials that are degenerate is as follows. Take $n(n-1)+1$ lines through the origin in $\bC^n$ with $n\ge 3$ such that no collection of $n$ of the lines is contained in a hyperplane. This means in particular that there do not exist local coordinates around the origin in which all of the lines are contained in the coordinate hyperplanes. Let $\mu:Y\to \bC^n$ be the composition of the blowup of the (strict transforms) of the lines, and let $f=1$. Then any log (very-)generic polynomial  $g$ in this setup will have singular locus equal to the union of the lines, and hence its singular locus cannot be contained in a union of coordinate axis. In particular, such $g$ cannot be non-degenerate.

(v)  Log (very-)generic polynomials are irreducible, whereas there exist reducible non-degenerate polynomials (e.g. $xy$). Non-degeneracy depends on the choice of coordinates, log (very-)genericity does not. Log (very-)genericity depends on the choice of the log resolution map $\mu$, and $\mu$ becomes a minimal log resolution for log (very-)generic polynomials $g$, in the sense of Remark \ref{rmkGg} (ii) for $f=1$.
A non-degenerate polynomial $g$ gives canonically a toric modification of $\bC^n$, corresponding to the normal fan to the  Newton polyhedron of $g$. This is usually not a log resolution of $g$, but any regular subdivision of the normal fan provides a log resolution $\mu$ of $g$. For such $\mu$, the strict transform of ${\rm{div}}(g)$ might not be ample, and this is an obstruction to log genericity even if $g$ is irreducible. Thus our theorems do not immediately imply the next result, suggested by the referee. However, our  method adapts to the canonical toric modification and can be translated in terms of combinatorics.
\end{rmk}

\begin{thm}\label{thmnondeg} Let  $\Gamma$ be a fixed Newton polyhedron such that all faces of $\Gamma$ not contained in coordinate hyperplanes are compact. Then the Monodromy Conjecture holds for all germs $g:(\bC^n,0)\to (\bC,0)$ of non-degenerate polynomials $g$ with Newton polyhedron $k\Gamma$ for all integers $k\gg 0$.
\end{thm}

\begin{rmk} (i) For complex analytic germs, the Monodromy Conjecture means that the poles of the  local topological zeta function at the origin, see Definition \ref{defTZ}, give eigenvalues of the monodromy at points close to the origin. The Newton polyhedron and non-degeneracy are taken at the origin, see \ref{subnondeg}.

(ii) A Newton polyhedron $\Gamma$ has the property that all faces not contained in coordinate hyperplanes are compact if and only if $k\Gamma$ has the same property for all $k\in\bZ_{>0}$. 
If $g$ has an isolated singularity at the origin, we can assume that its Newton polyhedron satisfies this property, since by finite determinacy we can add high-degree monomials $x_i^d$ in each variable $x_i$ to $g$ without changing the analytic type of the germ $g$.

(iii)  The Strong Monodromy Conjecture was proven in \cite{Lo2} for non-degenerate polynomials with Newton polyhedron $\Gamma$ satisfying that all faces not contained in coordinate hyperplanes are compact, and such that every compact facet (i.e. codimension-one face) satisfies a non-resonance condition. If a facet $\tau$ of $\Gamma$ is resonant, the dilated facet $k\tau$ is also resonant for $k\Gamma$ for any $k\in\bZ_{>0}$. However, resonance is allowed  in Theorem \ref{thmnondeg}.

\end{rmk}

The proof of Theorem \ref{thmMain} {\color{hot}$(a)$} relies on a formula for the monodromy zeta function associated to a tuple of polynomials due to Sabbah \cite{Sabb}, generalizing a classical result of A'Campo for $p=1$. The monodromy zeta functions recover the monodromy support of a tuple of polynomials by \cite{BLSW}, cf. Theorem \ref{mon}. The case $p=1$ of this fact was pointed out by Denef as a consequence of the perversity of the nearby cycles complex. Using this we show that all the candidates for polar hyperplanes of $Z_{\tilde{F}}^{top}$ arising from $\mu$ give components of the monodromy support of $\tilde{F}$. This also shows that the results in this note hold more generally for motivic zeta functions instead of topological zeta functions. Theorem \ref{thmMain2} {\color{hot}$(a)$} is a corollary of Theorem \ref{thmMain} {\color{hot}$(a)$}.

To address the parts {\color{hot} $(b)$} of the theorems, we show in Proposition \ref{propCand} that every candidate polar hyperplane of the relevant zeta functions is an actual polar hyperplane of order one. We prove then firstly Theorem \ref{thmMain2} {\color{hot}$(b)$} by adapting Loeser's proof from \cite{Lo2} that non-resonant compact codimension 1 faces of the Newton polytope of the germ of a non-degenerate hypersurface singularity give roots of the $b$-function. This method had also appeared in \cite{Li, L88} in the 1-dimensional case. The main differences with \cite{Lo2} stem from the fact that we do not assume compactness of exceptional divisors. We use a criterion  to produce roots of $b$-functions due to Hamm \cite{H} slightly improving a result of Malgrange \cite{Ma} used in \cite{Lo2}. Like \cite{Lo2}, we use a non-vanishing theorem for local systems of Esnault-Viehweg \cite{EV}. Theorem \ref{thmMain} {\color{hot}$(b)$} follows from Theorem \ref{thmMain2} {\color{hot}$(b)$} by using results for generalized Bernstein-Sato ideals of Gyoja \cite{Gy} and Budur \cite{B-ls}.

Even the case $f=1$ is intricate. This singles out the class of log (very-) generic hypersurfaces as an interesting class of singularities on its own. To stress that log very-generic polynomials have complicated singularities, we show in \ref{exGen} that the roots of the $b$-functions produced here are not necessarily negatives of jumping numbers.

In Section \ref{subNot} we fix notation. In Section \ref{secInv} we recall some facts about the objects of study. In Section \ref{secLG} we prove parts {\color{hot}$(a)$} of Theorems \ref{thmMain} and \ref{thmMain2}. In Section \ref{secLVG} we prove parts {\color{hot}$(b)$} of Theorems \ref{thmMain} and \ref{thmMain2}. Section \ref{secEx} contains further remarks and the proof of Theorem \ref{thmnondeg}.

\medskip
\noindent
{\bf Acknowledgement.} We thank M. Musta\c{t}\u{a}, L. Wu, P. Zhao, and the referee for useful comments. We are especially grateful to the referee for formulating Theorem \ref{thmnondeg}. The first author would like to thank MPI Bonn for hospitality during the writing of this article. The first author was partly supported by the grants STRT/13/005 and Methusalem METH/15/026 from KU Leuven, G097819N and G0F4216N from FWO (Research Foundation - Flanders).  The second author is supported by a PhD Fellowship from FWO.


\section{Notation.} \label{subNot}

For the proofs of the main results, we have to introduce some notation. With fix Setup \ref{eqDi}. We let $f=\prod_{i=1}^pf_i$ with $f_i\in\bC[x_1,\ldots, x_n]$.

\begin{subs}\label{subNot1} We set the following:

\begin{itemize}

\bigskip
\item $\bar{J}_{exc}$ is the set of irreducible components of  the exceptional locus of $\bbm$;
\item ${J}_{exc}=\{W\in \bar{J}_{exc}\mid W\cap Y\neq \emptyset\}$;

\bigskip
\item $F=(f_1,\ldots, f_p)$;
\item $\tilde{f}=f\cdot f_{p+1}$ with $f_{p+1}\in\bC[x_1,\ldots,x_n]$ such that $\mu$ is a log resolution for $\tilde f$ and  is an isomorphism over $\bC^n\setminus \tilde{f}^{-1}(0)$;
\item $\tilde{F}=(f_1,\ldots, f_p, f_{p+1})$ ;

\bigskip
\item $\bar{J}$ is the union of $\bar{J}_{exc}$ and the set of irreducible components of support of the divisor $\bbm^*(\text{div}(\tilde{f}))$ on $\bar{Y}$;
\item $J=\{W\in \bar{J}\mid W\cap Y\neq\emptyset\}$;

\bigskip
\item $W^\circ=W\setminus \cup_{W'\in \bar{J}\setminus\{W\}}W'$ for $W\in \bar{J}$;
\item $W_{J'}^\circ=(\cap_{W\in J'}W) \setminus (\cup_{W\in \bar{J}\setminus J'}W)$ for $J'\subset \bar{J}$.

\bigskip
\item For every $W\in\bar{J}$:
	\begin{itemize}
		\item[$\circ$] $n_W=\text{ord}_W(K_{\bbm})+1$, where $K_{\bbm}$ is the relative canonical divisor of $\bbm$;
		\item[$\circ$] $a_{i,W}=\ord_W(f_i)$;
		\item[$\circ$] $a_W=\ord_W(f)=\sum_{i=1}^pa_{i,W}$;
		\item[$\circ$] $N_W=\ord_W(\tilde{f})=\sum_{i=1}^{p+1}a_{i,W}.$
	\end{itemize}

\end{itemize}
\end{subs}

\begin{subs}\label{subNot2} In addition, we will consider the conditions:

\begin{itemize}

\bigskip
\item $L$ is a very ample line bundle on $\bbY$;
\item $H\in |L^{\otimes k}|$ is a general element and $k>0$ is an integer;
\item $f_{p+1}=g$ is a defining polynomial for $\mu(H\cap Y)$ in $\bC^n$;
\end{itemize}

\bigskip

This conditions are compatible with the assumption from \ref{subNot1} that $\mu$ is a log resolution of $\tilde f$ which is an isomorphism above $\bC^n\setminus\tilde f^{-1}(0)$.
We draw the attention that in this case some of the notions from \ref{subNot1} depend on the integer $k>0$, although we suppress this from the notation. 

\end{subs}

\section{Invariants of singularities}\label{secInv} 

We introduce the objects that form the subject of our results.  We use the setup and notation from \ref{subNot1}, but not yet the conditions from \ref{subNot2}. The following invariant was introduced by Denef-Loeser \cite{DLn}:

\begin{defi}\label{defTZ}
The {\it topological zeta function of $\tilde{F}$} is
$$
Z_{\tilde{F}}^{top}(s_1,\dots,s_{p+1})=\sum_{\emptyset\neq {J'}\subset J}\chi(W_{J'}^\circ)\prod_{W\in {J'}}\frac{1}{a_{1,W}s_1+\dots+a_{p+1,W}s_{p+1}+n_W},$$
where $\chi(\_)$ denotes the topological Euler characteristic. By replacing $W_{J'}^\circ$ with $W_{J'}^\circ\cap\mu^{-1}(x)$ one obtains the {\it local topological zeta function} $Z_{\tilde F,x}^{top}$ at a point $x\in\tilde f^{-1}(0)$. 
\end{defi}

It is known that the topological zeta function $Z_{\tilde{F}}^{top}$ can also be computed from any other log resolution of $\tilde f$.  If in addition the conditions from \ref{subNot2} hold,  $Z_{\tilde{F}}^{top}$ does depend on the log resolution $\mu$ of $f=f_1\cdot\ldots\cdot f_p$, since $f_{p+1}=g$ does; see also Remark \ref{rmkGg} (ii).

The support in $\bC^{p+1}$ of the divisor of poles of the rational function $Z_{\tilde{F}}^{top}$, the {\it polar locus}, is a hyperplane arrangement and will be denoted by $\mathcal{P}(Z_{\tilde{F}}^{top})$. The hyperplane $$\{a_{1,W}s_1+\dots+a_{p+1,W}s_{p+1}+n_W=0\}$$ is called the {\it candidate polar hyperplane from the component $W$}. 


\begin{defi}\label{defMS} The {\it monodromy support of $\tilde{F}$} is the subset 
$$\mathcal{S}_{\tilde{F}}\subset (\bC^*)^{p+1}$$
defined as follows. For each $\alpha\in (\bbc^*)^{p+1}$ one  has a rank one local system on $(\bC^*)^{p+1}$ with monodromy $\alpha_i$ around the $i$-th missing coordinate hyperplane. The pullback of this local system via $\tilde{F}|_{\bC^n\setminus \tilde{f}^{-1}(0)}: \bC^n\setminus \tilde{f}^{-1}(0)\to (\bbc^*)^{p+1}$ will be denoted $\cL_\al$. For $x\in \tilde{f}^{-1}(0)$, let $U_x$ be the complement of $\tilde{f}^{-1}(0)$ in a small open ball around $x$ in $\bC^n$. Then $\mathcal{S}_{\tilde{F}}$ consists of $\alpha\in (\bbc^*)^{p+1}$ for which there exist a point $x\in \tilde{f}^{-1}(0)$ with
$$
H^*(U_x, \mathcal L _\alpha)\ne 0.
$$
\end{defi}

An equivalent definition of $\cS_{\tilde{F}}$ is that this the support of the generalized monodromy action on the generalization of the nearby cycles complex \cite{Sabb}, by \cite{B-ls,BLSW}. The monodromy support $\cS_{\tilde{F}}$ is a finite union of torsion-translated codimension-one affine algebraic subtori of $(\bbc^*)^{p+1}$, by  \cite{BLSW}.

To a point $x\in \tilde{f}^{-1}(0)$ one associates the {\it monodromy zeta function} $Z_{\tilde{F},x}^{mon}$ of the stalk at $x$ of the generalized nearby cycles complex of $\tilde{F}$. We can take as definition the following formula  from \cite[Proposition 2.6.2]{Sabb},  \cite[Th\'eor\'eme 4.4.1]{Gui}, which recovers a classical formula due to A'Campo in the case $p=1$:

\begin{thm}\label{thmGui}

$$Z_{\tilde{F},x}^{mon}(t_1,\ldots, t_{p+1})=\prod_{W\in J} (t_1^{a_{1,W}}\cdots t_{p+1}^{a_{p+1,W}}-1)^{-\chi(W^\circ\cap \mu^{-1}(x))}.$$
\end{thm}

Denote by $$\mathcal{P}\mathcal{Z}(Z_{\tilde{F},x}^{mon})$$ the support of the divisor on $(\bC^*)^{p+1}$ associated to $Z_{\tilde{F},x}^{mon}$, the union of the zero and the polar locus, each being a finite union of torsion-translated codimension-one algebraic subtori.

Let $\Omega\subset \tilde{f}^{-1}(0)$ be a finite set consisting of general points of each stratum of a Whitney stratification of $\tilde{f}^{-1}(0)$. By \cite{BLSW} we have:

\begin{thm} \label{mon}
$$
\cS_{\tilde{F}}=\bigcup_{x\in \Omega}\mathcal{P}\mathcal{Z}(Z_{\tilde{F},x}^{mon}).
$$
\end{thm}

\section{Log generic polynomials}\label{secLG}

In this section we address log generic polynomials. For the proof of Theorem \ref{thmMain} {\color{hot}$(a)$}, we use the following estimate on asymptotic topological Euler characteristics: 

\begin{lemma}\label{lemmaChi}
Let $\bar{Y}$ be a smooth projective variety, $L$ a very ample line bundle on $\bar{Y}$, and $V\subset \bar{Y}$ a non-empty Zariski locally closed subset. Let $H\in |L^{\otimes k}|$ be a generic element for $k> 0$. Then, for $k\gg 0$,
$$
\chi(V\setminus H)=(-1)^{\dim {V}} \deg_L(\bar{V}_{top})\cdot k^{\dim V} + \text{ lower order terms in }k,$$
 where $\bar{V}_{top}$ is the union of the top-dimensional irreducible components of the closure of $V$. (If $\dim V=0$, there are no ``lower order terms in $k$", by convention.) In particular, for $k\gg 0$, 
$
(-1)^{\dim V}\chi(V\setminus H)>0.$

Moreover,
$$
\chi(V\cap H)=(-1)^{\dim {V}-1} \deg_L(\bar{V}_{top})\cdot k^{\dim V} + \text{ lower order terms in }k
$$
if $\dim V>0$, and in general
$$\chi(H\setminus V)=(-1)^{\dim {\bar{Y}-1}} \deg_L(\bar{Y})\cdot k^{\dim \bar{Y}} + \text{ lower order terms in }k.$$
\end{lemma}

\begin{proof}  If $\dim V=0$, then $V\cap H$ is empty, and the claims of the lemma are easily checked in this case. We assume from now that $\dim V>0$.

Assume first $V$ is closed and irreducible. In this case, $V\cap H$ is also irreducible and complete, hence
$$
\chi(V\cap H) = \int c_{SM}(V\cap H) 
$$
where $c_{SM}$ denote the Chern-Schwartz-MacPherson class, and the integral denotes the degree of a class in the Chow group, equal to the degree of the top-dimensional component of that class, see for example \cite[2.2]{A}. By \cite[Proposition 2.6]{A},
$$
c_{SM}(V\cap H)=\frac{H}{1+H}\cdot c_{SM}(V).
$$
Since $\dim V>0$, the right-hand side can be expanded  to yield the equality 
$$
\chi(V\cap H) = (-1)^{\dim V -1}k^{\dim V}\int c_1(L)^{\dim V}\cdot c_{SM}(V)_{\dim V}  + \text{ lower order terms in }k
$$
(if $\dim V$ would be 0, this is not true since the right-hand side is not zero whereas the class $(H/(1+H)) c_{SM}(V)$ is zero by convention).

By \cite[Theorem 1.1]{A},
$$
\int c_1(L)^{\dim V}c_{SM}(V)_{\dim V} =  \deg_{L}(V)>0.
$$
Thus
$$
\chi(V\cap H) =(-1)^{\dim V -1}\deg_L(V)\cdot k^{\dim V} + \text{ lower order terms in }k.
$$

We will extend this result now to the case when $V$ is closed but not necessarily irreducible.  Let $V_i$ with $i\in I$ be the irreducible components of $V$. Using inclusion-exclusion, we can write
$$
\chi(V\cap H) = \sum_{i\in I}\chi(V_i\cap H) + \sum_{W}m_W\cdot\chi(W\cap H)
$$
where $W$ are irreducible components of intersections of at least two distinct irreducible components of $V$, and $m_W$ are suitable multiplicities independent of $k$ since $H$ is generic. It follows by the first part of this proof that 
\begin{align*}
\chi(V\cap H) &= (-1)^{\dim V -1} \sum_{\substack{i\in I\\\dim V_i=\dim V}}\deg_L(V_i)\cdot k^{\dim V} + \text{ lower order terms in }k\\
&  =(-1)^{\dim {V} -1} \deg_L(V_{top})\cdot k^{\dim V} + \text{ lower order terms in }k 
\end{align*}
where $V_{top}$ is the union of the top-dimensional irreducible components of $V$.

Now let $V$ be locally closed and denote by $\bar{V}$ the closure of $V$ in $\bar{Y}$. Let $Z=\bar{V}\setminus V$, so that $Z$ is closed in $\bar{Y}$. Then
\begin{align*}
\chi(V\setminus H) &= \chi(\bar{V}\setminus H) - \chi (Z\setminus H)\\
&= \chi(\bar{V}) - \chi(\bar{V}\cap H) - \chi(Z) + \chi(Z\cap H)\\
&= (-1)^{\dim {V}} \deg_L(\bar{V}_{top})\cdot k^{\dim V} + \text{ lower order terms in }k,
\end{align*}
where the last equality follows from the case handled above. This proves the first assertion.

Writing
$$
\chi(V\cap H) = \chi (\bar{V}\cap H) - \chi (Z\cap H),
$$
we obtain the second assertion, since $\dim (Z\cap H) < \dim V$ if $\dim V>1$, and if $\dim V=1$ then $Z=\emptyset$ by the genericity of $H$.

Next, \begin{align*} \chi(H\setminus V)& = \chi(H) -\chi(\bar{V}\cap H) + \chi(Z\cap H)\\&= (-1)^{\dim {\bar{Y}-1}} \deg_L(\bar{Y})\cdot k^{\dim \bar{Y}} + \text{ lower order terms in }k,\end{align*}as claimed.
\end{proof}

\begin{subs}\label{sub17a} {\bf Proof of Theorem \ref{thmMain} {\color{hot}$(a)$}.} We use the notation from  \ref{subNot1} together with the conditions from \ref{subNot2}, and take $k\gg 0$.
We show that the Monodromy Conjecture holds for the tuple $\tilde{F}=(f_1,\ldots,f_{p+1})$. 

It is enough to show that every candidate polar hyperplane for $Z_{\tilde{F}}^{top}$ arising from $\mu$ is mapped via the exponential map into the monodromy support of $\tilde{F}$. That is, that
$$
Exp(\{a_{1,W}s_1+\ldots+a_{p+1,W}s_{p+1}+n_W=0\}) \subset \cS_{\tilde{F}}
$$
for every $W\in J$. 

By Theorem \ref{mon}, it is thus enough to show that for every $W\in J$, the locus $$\left\{\prod_{i=1}^{p+1}t_i^{a_{i,W}}=1\right\}\subset(\bC^*)^{p+1}$$ is contained in $\mathcal P\mathcal Z (Z_{\tilde{F},x}^{mon})$ for some $x\in \Omega$.

For $W\in J$ and $x\in \Omega$, let
$$
W^\circ _{x}=W^\circ\cap\mu^{-1}(x),
$$
$$J_x=\{W\in J\mid W^\circ _{x}\neq\emptyset\}.$$ 
Then
\be\label{eqFZ}
Z_{\tilde{F},x}^{mon}(t_1,\ldots, t_{p+1})=\prod_{W\in J_x} (t_1^{a_{1,W}}\cdots t_{p+1}^{a_{p+1,W}}-1)^{-\chi(W^\circ _{x})}
\ee
by Theorem \ref{thmGui}.

If $W=H$ and $x\in \Omega$ is a general point on $\mu(H\cap Y)$, the vector $a_{\bullet, W}$ in $\bZ^{p+1}$ is equal to $(0,\ldots, 0,1)$ and $$\chi(W^\circ\cap \mu^{-1}(x))=\chi(\{x\})=1.$$
Moreover, $J_{x}=\{H\}$ in this case. Thus
$
Z_{\tilde{F},x}^{mon}=(t_{p+1}-1)^{-1}
$
and so $\{\prod_{i=1}^{p+1}t_i^{a_{i,W}}=1\}=\mathcal P\mathcal Z (Z_{\tilde{F},x}^{mon})$, which proves the claim in this case.

For the remaining cases fix $x\in\Omega$ a general point of a stratum, different from above, of a Whitney stratification of $\tilde f^{-1}(0)$. It is enough to show that the locus $\{\prod_{i=1}^{p+1}t_i^{a_{i,W}}=1\}$ is contained in $\mathcal P\mathcal Z (Z_{\tilde{F},x}^{mon})$ for every $W$ in $J_x\setminus\{H\}.$ For such $W$, let $$W_x= (W\setminus \cup_{W'\in \bar{J}\setminus\{W,H\}}W') \cap \mu^{-1}(x),$$ so that 
$$
 W^\circ_{x}=W_x \setminus H.
$$
Then
$$
\chi(W^\circ _{x}) = (-1)^{\dim W^\circ _{x}}\deg_L((\ol{W_x})_{top})\cdot k^{\dim W_x} +\text{ lower order terms in }k,
$$
by Lemma \ref{lemmaChi}, where $(\ol{W_x})_{top}$ is the union of the top-dimensional irreducible components of the Zariski closure of $W_x$, and there are no ``lower order terms in $k$" if $\dim W_x =0$. In particular, 
\be\label{eqG1}\chi(W^\circ _{x})\neq 0\quad\text{ for }\quad k\gg 0,\ee 
and hence every $W\in J_x\setminus\{H\}$ contributes to the right-hand side of (\ref{eqFZ}) with a non-trivial factor before cancellations.

Suppose that a non-trivial irreducible factor $P(t)$ of  $\prod_{i=1}^{p+1}t_i^{a_{i,W}}-1$ for some $W\in J_x\setminus\{H\}$ cancels out from (\ref{eqFZ}) and the zero locus of $P(t)$ does not lie in $\mathcal P\mathcal Z (Z_{\tilde{F},x}^{mon})$. Let $J'\subset J_x\setminus\{H\}$ be the set of all $W\in J_x\setminus\{H\}$  with strictly positive multiplicity  of $P(t)$ as a factor of $\prod_{i=1}^{p+1}t_i^{a_{i,W}}-1$. Since the latter polynomial is reduced, this multiplicity has to equal 1. 
The cancellation then implies  
$$\sum_{W\in J'}  \chi(W^\circ _{x}) = 0.
$$
Let
$$
r=\max\{\dim W^\circ_{x}\mid W\in J'\}.
$$
Then
$$
0=\sum_{W\in J', \, \dim W^\circ_{x}=r} (-1)^{r} \deg_L((\ol{W_x})_{top})\cdot k^{r} +\text{ lower order terms in }k,
$$
where for $r=0$ there are no ``lower order terms in $k$". Since  the degree of a non-empty set is $>0$, and hence the coefficient of $k^r$ is non-zero, this implies that
\be\label{eqG2}
\sum_{W\in J'}  \chi(W^\circ _{x}) \neq 0\quad\text{ for }\quad k\gg 0,
\ee
which is a contradiction. $\hfill\Box$
\end{subs}

\begin{subs}\label{subBy} {\bf Proof of Theorem \ref{thmMain2} {\color{hot}$(a)$}.} We let $f_{p+1}=g$ and $\tilde{f}=f\cdot f_{p+1}$ as in the proof of Theorem \ref{thmMain} {\color{hot}$(a)$}, for $k\gg 0$. Since $Z_{\tilde{f}}^{top}(s)=Z_{\tilde{F}}^{top}(s,\ldots, s)$, the restriction of the polar locus of $Z_{\tilde{F}}^{top}$ to the line $s_1=\ldots=s_{p+1}=s$ contains the polar locus of  $Z_{\tilde{f}}^{top}$. The conclusion then follows from Theorem \ref{thmMain} {\color{hot}$(a)$} and the fact that the restriction of the monodromy support $\mathcal S_{\tilde{F}}$ of $\tilde{F}$ to $s_1=\ldots=s_{p+1}=s$ equals the monodromy support $\mathcal S_{\tilde{f}}$ of $\tilde{f}$,  by \cite[Theorem 2.11]{BLSW}. 
$\hfill\Box$
\end{subs}

\section{Log very-generic polynomials}\label{secLVG}

In this section we address log very-generic polynomials. We fix Setup \ref{eqDi}.

\begin{subs}\label{subsAL} {\bf The log very-genericity condition.} Since $\bar{\mu}$ is a composition of blowing ups of smooth closed subvarieties, 
$$
 \bZ\oplus \bigoplus_{W\in \bar{J}_{exc}} \bZ [W] \xa{\sim} \text{Pic}(\bar{Y}),\quad (d,b_W)\mapsto {\bbm}^*\cO_{\bP^n}(d)\otimes\cO_{\bar{Y}}\left(-\sum_{W\in \bar{J}_{exc}} b_W W\right)
$$
is an isomorphism of finitely generated abelian groups. We let $\cA\subset \text{Pic}(\bbY)$ be the subset of ample isomorphism classes. Then the above isomorphism identifies $\cA$ with a subcone
$$\cA\subset \bR_+\oplus \bigoplus_{W\in \bar{J}_{exc}} \bR_+ [W],$$
that is, if $L\in\cA$ then every integral point in the ray $\bR_+L$ belongs to $\cA$, where $\bR_+$ denote the strictly positive real numbers.


We introduce the subcone $\cA^{vg}$ of $\cA$ used in Definition \ref{defVGb} of log very-general polynomials:

\begin{defi}\label{defAvg}
Let $\cA^{vg}\subset \cA$ be the set of isomorphism classes of ample line bundles on $\bar{Y}$ such that
for each $W\in J_{exc}$, 
$$
\frac{n_W}{b_W}b_{W'} \not\in\bZ
$$
for all $W'\in \bar{J}_{exc}\setminus \{W\}$ with $W\cap W'\neq\emptyset$. 
\end{defi}

Note that the condition defining $\cA^{vg}$ in $\cA$ is actually a condition on the $\bbm$-ample cone, which coincides with the image of the projection of $\cA$ to the space indexed by $\bar{J}_{exc}$.

\begin{lemma}\label{lemPSBC}

The subset $\cA^{vg}$ of $\cA$ is a non-empty subcone.
\end{lemma}
\begin{proof} By definition $\cA^{vg}$ is a subcone if non-empty. We show that it is non-empty. Fix $L\in \cA$ ample with associated coordinates $b_W$ for $W\in \bar{J}_{exc}$.

Choose integers $p_W\gg 0$ for each $W\in \bar{J}_{exc}$ such that $n_Wp_W/p_{W'}\not\in \bZ$ for all pairs $(W, W')$ of different elements in $\bar{J}_{exc}$. This is possible since $\bar{J}_{exc}$ is finite. Moreover,
$$
L-\sum_{W\in \bar{J}_{exc}}\frac{1}{p_W}W
$$
is an ample $\bQ$-divisor class by \cite[Example 1.3.14]{La}.  Let $p=\prod_{W\in \bar{J}_{exc}}p_W$. Thus
$$
p(L-\sum_{W\in \bar{J}_{exc}}\frac{1}{p_W}W)
$$
is an ample integral divisor class. Replacing $L$ by this new divisor class, one replaces $b_W$ by $p(b_W+1/p_W)$ for each $W\in \bar{J}_{exc}$. Moreover,
$$
\frac{n_W}{p(b_W+1/p_W)} p(b_{W'}+1/p_{W'}) = \frac{p_Wn_W}{p_Wb_W+1}\cdot \frac{p_{W'}b_{W'}+1}{p_{W'}}
$$
is not an integer since $p_{W'}$ does not divide the numerator. This proves the claim.
\end{proof}

\end{subs}

\begin{subs}{\bf Candidate vs. actual poles.} We  keep the notation from \ref{subNot1} together with the conditions from \ref{subNot2}, and let $k\gg 0$. In particular, $\tilde{F}=(f_1,\ldots,f_{p+1})$, with $f=\prod_{i=1}^p$, $\tilde{f}=ff_{p+1}$, and $f_{p+1}=g$ is log generic.

\begin{prop}\label{propCand}
If $f_{p+1}=g$ is log very-generic, then:

\begin{enumerate}[(i)]
\item Every candidate polar hyperplane of $Z_{\tilde{F}}^{top}(s_1,\ldots, s_{p+1})$ arising from the exceptional locus of $\mu$ is a polar hyperplane of order one.
\item  Every candidate pole of $Z_{\tilde{f}}^{top}(s)$ arising from the exceptional locus of $\mu$ is a pole of order one.\end{enumerate}
\end{prop}

\begin{proof}   We prove $(ii)$ first. We have
\be\label{eqZW}
Z_{\tilde{f}}^{top}(s)=\sum_{\emptyset\neq J'\subset J}\chi(W_{J'}^\circ)\prod_{W\in J'}\frac{1}{N_Ws+n_W}.
\ee
Moreover,
$$
N_W=a_{W}+ a_{p+1,W}
$$
where
$$a_{p+1,W}=\ord_{W}(\mu^*(\mu(H\cap Y))).$$
On the other hand,
$$H\cap Y =\mu^*(\mu(H\cap Y)) - \sum_{W\in J_{exc}}a_{p+1,W}(W\cap Y)$$ 
as a divisor on $Y$, and $\cO_{Y}(H\cap Y)\simeq L^{\otimes k}|_Y$ by definition of $H$. Let
$$
L\simeq {\bbm}^*\cO_{\bP^n}(d)\otimes\cO_{\bar{Y}}\left(-\sum_{W\in \bar{J}_{exc}} b_W W\right)
$$
be the unique representation of the isomorphism class of $L$ in the cone $\cA$. Then 
\be\label{eqbW}
a_{p+1,W}=\left\{
\begin{array}{cl}
kb_W  & \text{ for } W\in J_{exc},\\
1& \text{ for }W=H,\text{ in which case }a_W=0,\\
0&\text{ for }W\in J\setminus (J_{exc}\cup\{H\}).\\
\end{array}\right.
\ee

From now on we will now restrict to those $L$ in $\cA^{vg}$ as in Definition \ref{defAvg}, this being the reason why we prove the proposition only log very-generically and not log generically.




We show that the candidate pole from $W\in J_{exc}$,
$$-\frac{n_W}{N_W}=-\frac{n_W}{a_W+kb_W},$$ is a pole of order one of $Z^{top}_{\tilde{f}}(s)$ for $k\gg 0$. 

Firstly, the pole order is at most 1. If the order would be  $>1$, then from formula (\ref{eqZW}) we see that there must exist $W'\in J\setminus\{W\}$ with $W\cap W'\neq \emptyset$ such that
$$
\frac{n_W}{N_W	}=\frac{n_W}{a_W+kb_W}=\frac{n_{W'}}{N_{W'}}.
$$
This is impossible for $W'\not\in J_{exc}$ for large $k$. If this equality happens for $W'\in J_{exc}$ for infinitely many $k\in\bN$, by taking limit as $k$ goes to infinity we obtain that
$$
\frac{n_W}{b_W}b_{W'}= n_{W'}
$$
which is excluded by the condition that $L$ is in $\cA^{vg}$. 

Since the pole order is at most $1$, to show that the order is equal to 1 it is enough to show that the evaluation at  $s=-n_W/N_W$ of $$(N_Ws+n_W)Z^{top}_{\tilde{f}}(s)$$ is a non-zero number for $k\gg 0$.  By (\ref{eqZW}), this number is
$$
\sum_{\emptyset \neq J'\subset J\setminus\{W\}} \chi(W_{J'\cup\{W\}}^\circ) \prod_{W'\in J'} \frac{N_W}{N_Wn_{W'}-N_{W'}n_W}. 
$$
The denominators are all different than zero if $W^\circ_{J'\cap\{W\}}\neq\emptyset$, as we have seen already. 
For $k\gg 0$,  using (\ref{eqbW}) and the asymptotic behaviour from Lemma \ref{lemmaChi} for the Euler characteristics, the dominant term corresponds to $J'=\{H\}$ and it is equal to $(-1)^{n-1}\deg_L(W) k^{n-1}$. Since this term is non-zero, this proves (ii).


Now we show $(i)$. Let $W\in J_{exc}$. The candidate polar hyperplane for $Z_{\tilde{F}}^{top}$ from $W$ is $\{\sum_{i=1}^{p+1}a_{i,W}s_i+n_W=0\}$. Note that $N_W=\sum_{i=1}^{p+1}a_{i,W}$. Since $Z_{\tilde{f}}^{top}(s)=Z_{\tilde{F}}^{top}(s,\ldots, s)$, the restriction of the polar locus of $Z_{\tilde{F}}^{top}$ to the line $s_1=\ldots=s_{p+1}=s$ contains the polar locus of  $Z_{\tilde{f}}^{top}(s)$. By part (ii), $W$ contributes with the pole
$$
-\frac{n_W}{N_W}=\left.\left\{n_W+\sum_{i=1}^{p+1}a_{i,W}s_i=0\right\}\right | _{s_1=\ldots=s_{p+1}=s} 
$$
to $Z_{\tilde{f}}^{top}(s)$, and we have seen that our assumptions imply that $\frac{n_W}{N_W}\neq \frac{n_{W'}}{N_{W'}} $ for every $W'\in J\setminus \{W\}$ for $k\gg 0$. Thus the polar locus of $Z_{\tilde{F}}^{top}$ must contain the candidate from $W$.
\end{proof}


\end{subs}

\begin{subs}\label{subSMC}{\bf Strong Monodromy Conjecture.} We fix, as always in this section, the setup and notation from \ref{subNot1} together with the conditions from \ref{subNot2}, and take $k\gg 0$. We will show:

\begin{prop}\label{propBSZ} If $f_{p+1}=g$ is log very-generic:
\begin{enumerate}[(i)]
\item Every polar hyperplane of $Z^{top}_{\tilde{F}}(s_1,\dots,s_{p+1})$ arising from the exceptional locus of $\mu$ is an irreducible component of $Z(B_{\tilde{F}})$.
\item Every pole of $Z_{\tilde{f}}^{top}(s)$ arising from the exceptional locus of $\mu$ is a root of the $b$-function of $\tilde{f}$.
\end{enumerate}
\end{prop}


Granted this proposition, we can complete:

\begin{proof}[Proof of parts {\color{hot}$(b)$} of Theorems \ref{thmMain} and \ref{thmMain2}.] The non-exceptional components contribute trivially with irreducible components of the zero locus of the Bernstein-Sato ideal (resp. with roots of the $b$-function), by localizing around general point, hence smooth, on such a component. Thus the claim follows from the previous proposition.  
\end{proof}

The rest of the section is dedicated to the proof of the proposition. 

\begin{proof}[Proof of Proposition \ref{propBSZ} {\color{hot} (ii)}]
 Let $W\in J_{exc}$. We prove that $-{n_W}/{N_W}$ is a root of the $b$-function of $\tilde{f}$ for $k\gg 0$.

 If $W\cap Y$ is compact, one can apply \cite[6.6]{Lo2} directly. However, $W\cap Y$ is typically not compact, so we have to adapt the proof of \cite[6.6]{Lo2}. 
 
We will use \cite[III]{H} as a bridge from the geometric data to roots of the $b$-function. Thus the essential task  is to construct a horizontal multi-valued family of cycles
$$
\gamma(t) \in H_{n-1}(\tilde{f}^{-1}(t),\bC)
$$
for small $t\neq 0$, such that
$$
\lim_{t\ra 0} \;t^{1-\frac{n_W}{N_{W}}} \int_{\gamma(t)} \frac{dx_1\wedge \ldots\wedge dx_n}{d\tilde{f}}
$$ 
exists and is a non-zero constant.

The proof takes a few steps:

\medskip

- Step 1: In general, a multi-valued form $
\tilde{f}^{\al}dx_1\wedge\ldots\wedge dx_n
$ pulls back to a twisted logarithmic form $\omega$ on $\bar Y$. We use the numerical data on the log resolution to identify  twisting line bundles for $\omega$ and for its associated residue form $\eta$.

\medskip

- Step 2: We specialize to $\al=-\frac{n_W}{N_{W}}$. We show that the twisting bundles are non-resonant and positive enough to apply the main theorem of \cite{EV} and obtain that the form $\eta$ determines a non-zero homology class $\gamma$ of an associated (dual) local system. By construction $\int_\gamma\eta\neq 0$.

\medskip

- Step 3: As in the proof of \cite[6.6]{Lo2}, using a trivializing finite cover for the associated local system, we construct the multi-valued family of cycles $\gamma(t)$ as above.

\medskip

- Step 4: Apply \cite[III]{H}.

\medskip

We now give the details of the proof.

\medskip

{\it Step 1.} Let $\al\in\bR$. The multi-valued form
$$
\tilde{f}^{\al}dx_1\wedge\ldots\wedge dx_n
$$
gives a global section in
$
\Gamma(U,\Omega^n_U\otimes_\bC \cL_U )
$
where  $U=\bC^n\setminus \tilde{f}^{-1}(0)$, $\Omega^n_U$ is the sheaf of holomorphic $n$-forms, and $\cL_U$ is the rank one local system on $U$ defined as the pullback via $\tilde{f}$ of the rank one local system on $\bC^*$ with monodromy multiplication by $\exp(-2\pi i \al)$ around the origin.

By construction of $\tilde{f}$, $\mu$ is an isomorphism over $U$,
$$
V:=Y\setminus (\tilde{f}\circ\mu)^{-1}(0)\xrightarrow{\sim} U.
$$ 
We have
$$
\omega=\mu^*(\tilde{f}^{\al}dx_1\wedge\ldots\wedge dx_n)\in \Gamma(V,\Omega^n_V\otimes_\bC \cL_V )
$$
where $\cL_V=\mu^{-1}\cL_U$. For every $W'\in {\bar{J}}$, the order of vanishing of $\omega$ along $W'$ is well-defined and
\be\label{eqOrd}
\ord_{W'}(\omega)=n_{W'}-1+\al N_{W'}.
\ee
The monodromy of $\cL_V$ around $W'$ is $$\exp(-2\pi i\cdot\ord_{W'}(\omega))=\exp(-2\pi i \al N_{W'}).$$ Moreover, $\omega$ has a meromorphic extension to $\bar{Y}$ across the simple normal crossings divisor 
$$A= \bar{Y}\setminus V=\sum_{W'\in \bar{J}}W'.$$ 
More precisely,
$$
\omega\in\Gamma(\bar{Y}, \Omega^n_{\bar{Y}}(\log A)\otimes_{\cO_{\bar{Y}}}\mathcal{M}) \subset \Gamma(V,\Omega^n_V\otimes_\bC \cL_V )
$$
where 
$$
\mathcal{M}=\cO_{\bar{Y}}\left(-\sum_{W'\in \bar{J}}(\ord_{W'}(\omega)+1)\cdot W'\right)
$$
is defined (a definition is necessary since the coefficients are in $\bR$) as
$$
\mathcal{M}=\cL_Y^{can}\otimes_{\cO_{\bar{Y}}}\cO_{\bar{Y}}\left(-\sum_{W'\in \bar{J}}\lfloor \ord_{W'}(\omega)+1\rfloor \cdot W'\right),
$$
with
$\cL_V^{can}$ the canonical Deligne extension of $\cL_V$, and $\lfloor \_ \rfloor$ denoting the round-down. Recall that $\cL_V^{can}$ is a line bundle on $\bar{Y}$ extending $\cO_V\otimes_\bC\cL_V$ and
 is defined as follows. Around a general point of $W'$, let $z$ be a local holomorphic function on $\bar{Y}$ defining $W'$, and let $u$ be a local multi-valued frame for $\cL_V$. Then $$\cL_V^{can}=\cO_{\bar{Y}}\left(-\sum_{W'\in \bar{J}}\{\ord_{W'}(\omega)+1\}\cdot W'\right)$$ is defined by declaring $z^{\{\ord_{W'}(\omega)+1\}}u$ to be a local holomorphic frame, that is, locally $$\cL_V^{can}\simeq \cO_{\bar{Y}}\cdot z^{\{\ord_{W'}(\omega)+1\}}u,$$ where $\{\_\}$ denotes the fractional part.

By definition of $\mathcal{M}$, the $\mathcal{M}$-valued log differential form $\omega$ has no poles nor zeros on $V$. Therefore $\omega$ induces an isomorphism of invertible sheaves
$$
\cO_{\bar{Y}}\xrightarrow{\sim}\Omega^n_{\bar{Y}}(\log A)\otimes_{\cO_{\bar{Y}}}\mathcal{M},
$$
and hence an isomorphism
$$
\mathcal{M}^{-1}\xrightarrow{\sim}\Omega^n_{\bar{Y}}(\log A).
$$

Let us denote the residue of $\omega$ along $W$ by
$$
\eta\in \Gamma(W,\Omega^{n-1}_W(\log A_W)\otimes_{\cO_{\bar{Y}}} \mathcal{M})
$$ 
where
$$
A_W=(A-W)|_W
$$ 
so that 
$W^\circ=W\setminus A_W.$
By definition,  $\eta$ is locally $(z\omega/dz)|_{W^\circ}$  where $z$ is a holomorphic function on $\bar{Y}$ defining $W$.

{\it Step 2.} We take from now $$\al=-\frac{n_W}{N_{W}}.$$ The effect of this choice is that
$$
\ord_W(\omega) + 1 = 0.
$$
This implies that $\eta\neq 0$
and
\be\label{eqc1}
c_1(\cL_V^{can}|_W) = -\sum_{W'\in \bar{J}_W} \{\ord_{W'}(\omega)+1\}\cdot[W'|_W]\quad \in H^2(W,\bR)
\ee
where
$$
\bar{J}_W:=\{W'\in \bar{J}\mid W\neq W'\text{ and }W\cap W'\neq\emptyset\}.
$$
Equation (\ref{eqc1}) guarantees that there exists a rank one local system $\cL$ on $W^\circ$ with monodromy around $W'|_W$ with $W'\in \bar{J}_W$ precisely
$$
-\exp(2\pi i \{\ord_{W'}(\omega)+1\}),
$$
by applying for example \cite[Theorem 1.2 and \S 3]{B-uls}. Then the canonical Deligne extension of $\cL$ to $W$ is
$$
\cL^{can}= \cL_V^{can}|_W. 
$$ 
Thus
$$
\eta\in \Gamma(W,\Omega^{n-1}_W(\log A_W)\otimes_{\cO_{W}} \mathcal{M})\subset \Gamma(W^\circ, \Omega_{W^\circ}^{n-1}\otimes_\bC\cL).
$$
That is, $\eta$ is a meromorphic $\cL$-twisted differential form with no poles nor zeros on $W^\circ$, and $\mathcal{M}|_W$ is the smallest invertible sheaf on $W$ with this property. This is the first ingredient needed to apply \cite{EV}.

Since $H\cap W$ is an irreducible component of $A_W$ and is a very ample divisor class on $W$ for $k\gg 0$, one has that
\be\label{eqCx}
\Omega^{n-1}_W(\log A_W)\simeq \cO_W(K_W+A_W)\simeq \cO_W(B)\otimes_{\cO_{W}}L^{\otimes k}
\ee
is nef and big for $k\gg 0$ by \cite[Example 1.2.10]{La}, where $B=K_W+A_W-(H\cap W)$. This is the second ingredient needed to apply \cite{EV}.

We now assume further that $L\in \cA^{vg}$ as in Definition \ref{defAvg}, this being the reason why we prove the proposition only log very-generically and not log generically. This and (\ref{eqbW}) imply for $k\gg 0$ that for all $W'\in \bar{J}_W$,
\be\label{eqSMX}
\frac{n_W}{N_{W}} N_{W'}\not\in \bZ,
\ee
or, equivalently by (\ref{eqOrd}),
$$
\ord_{W'}(\omega)\not \in \bZ.
$$
Thus none of the monodromies of $\cL$ is $1$. This is the third and last ingredient needed to apply \cite{EV}.

We can now apply the main theorem of \cite{EV} and obtain that the form $\eta$ determines a non-zero class in $H^{n-1}(W^\circ,\cL)$.  Since $W\setminus (H\cap W)$ is affine, its subset $W^\circ$ is also affine. It follows by \cite[(1.5)]{EV1} and its proof that
$
H^{n-1}(W^\circ,\cL) = H^{n-1}_c(W^\circ,\cL).
$
Therefore there exists a cycle
$$
\gamma\in H_{n-1}(W^\circ,\cL^\vee)
$$
with coefficients in the dual local system of $\cL$, such that
$$
\int_\gamma\eta \neq 0. 
$$

{\it Step 3.} For this step the arguments are as in the proof of \cite[6.6]{Lo2}. Consider a Gelfand-Leray form
$$
\frac{\tilde{f}^{1-\frac{n_W}{N_W}}dx_1\wedge \ldots\wedge dx_n}{d\tilde{f}}
$$
on $U$. A local computation shows that the pullback by $\mu^*$ extends over $W^\circ$ and
$$
\mu^*\left .\left(\frac{\tilde{f}^{1-\frac{n_W}{N_W}}dx_1\wedge \ldots\wedge dx_n}{d\tilde{f}} \right)\right |_{W^\circ} =\eta
$$
up to multiplication by a non-zero constant.

Let $N$ be the lowest common multiple of all $N_{W'}$ for $W'\in \bar{J}$. Let $\tilde{Y}\ra\bC$
be the normalization of the base change of $\tilde{f}\circ\mu:Y\ra \bC$ by the morphism $\bC\ra\bC$, $t\mapsto t^N$. Let $\tilde{W}^\circ$ be the inverse image of $W^\circ$ in $\tilde{Y}$. Then the natural map $$\nu:\tilde{W}^\circ\ra W^\circ$$ is \'etale and  $\nu^*\cL$ is the constant sheaf. Thus 
$
\nu^*(\eta)\in H^{n-1}(\tilde{W}^\circ,\bC)
$
and one has a cycle
$
\tilde{\gamma}\in H_{n-1}(\tilde{W}^\circ,\bC)
$
such that
$$
\int_{\tilde{\gamma}} \nu^*\eta \ne 0.
$$
Since $\tilde{W}^\circ$ is smooth, ${\tilde{f}}$ lifts to a trivial fibration on a small tubular neighborhood $T$ of $\tilde{W}^\circ$ in $\tilde{Y}$. Let $T_t$ be the fibers for small $t$. By parallel transport, $\tilde{\gamma}=\tilde{\gamma}(0)$ for a horizontal family   
$$\tilde{\gamma}(t)\in H_{n-1}(T_t,\bC)$$
for small $t$. Pushing forward to $\bC^n$, we obtain a horizontal multi-valued family of cycles
$$
\gamma(t) \in H_{n-1}(\tilde{f}^{-1}(t),\bC)
$$
for small $t\neq 0$, such that
$$
\lim_{t\ra 0} \;t^{1-\frac{n_W}{N_{W}}} \int_{\gamma(t)} \frac{dx_1\wedge \ldots\wedge dx_n}{d\tilde{f}}
$$ 
exists and is a non-zero constant.

{\it Step 4.} As already mentioned, the last statement implies directly that $-n_W/N_{W}$ is a root of the $b$-function of $\tilde{f}$ by applying \cite[III]{H}.  This finishes the proof of Proposition \ref{propBSZ} {\color{hot} (ii)}.
\end{proof}

\begin{proof}[Proof of Proposition \ref{propBSZ} {\color{hot}$(i)$}.]

Let $W\in J_{exc}$. Since $k\gg 0$, we have $a_{W,p+1}>n_W$.
Denote by $\mathfrak{L}$ the set of all $\mathbf{l}=(l_1,\dots,l_{p+1})\in\mathbb{Z}_{>0}^{p+1}$ such that $$\frac{n_W}{l_1a_{W,1}+\dots+l_{p+1}a_{W,p+1}}(l_1a_{W',1}+\dots +l_{p+1}a_{W',p+1})\not\in\mathbb{Z}$$
 for all $W'\in \bar{J}\setminus\{W\}$ with $W'\cap W\neq \emptyset$.  
Then for all  $(l_1,\dots,l_p)\in \mathbb{Z}_{>0}^p$ and for $l_{p+1}\gg0$ large enough relative to $(l_1,\dots,l_p)$, 
$$(l_1,\dots,l_{p+1})\in\mathfrak{L}.$$
To see this, one can use (\ref{eqbW}), the constraint from Definition \ref{defAvg}, and take limits as $l_{p+1}\to \infty$.

It follows from the proof of Proposition  \ref{propBSZ} {\color{hot} (ii)} that for $\mathbf{l}\in\mathfrak{L}$, $$r_{W,\mathbf{l}}:=-\frac{n_W}{l_1a_{W,1}+\dots+l_{p+1}a_{W,p+1}}$$ is a root of the $b$-function of $f_1^{l_1}\dots f_{p+1}^{l_{p+1}}$. 
This easily implies that for all $b(s_1,\dots,s_{p+1})\in B_{\tilde{F}}^{\,\mathbf{l}}$,
$$b(l_1r_{W,\mathbf{l}},\dots,l_{p+1}r_{W,\mathbf{l}})=0,$$
cf. \cite[Lemma 4.20]{B-ls}. Here $B_{\tilde{F}}^{\,\mathbf{l}}$ is the generalized Bernstein-Sato ideal consisting of $b\in \bC[s_1,\ldots,s_{p+1}]$ such that
$$
b\prod_{i=1}^p f_i^{s_i}=P\prod_{i=1}^{p+1}f_i^{s_i+l_i}
$$
for some $P\in\sD[s_1,\ldots,s_{p+1}]$.

Denote
$$q_{W,\mathbf{l}}:= (l_1r_{W,\mathbf{l}},\dots,l_{p+1}r_{W,\mathbf{l}})\in\mathbb{C}^{p+1},$$
so that we can write
$$\mathfrak{Q}:=\{q_{W,\mathbf{l}}\mid \mathbf{l}\in\mathfrak{L} \} \subset \bigcup_{\mathbf{l}\in\mathfrak{L}}Z(B_{\tilde{F}}^{\,\mathbf{l}})=:{\mathcal{Z}}.$$
Notice that all $q_{W,\mathbf{l}}$ lie on the  polar hyperplane of $Z_{\tilde{F}}^{top}(s_1,\dots,s_{p+1})$ contributed by $W$,
$$L:=\{a_{W,1}s_1+\dots+a_{W,p+1}s_{p+1}+n_W=0\}.$$

Denote by $t_i:\mathbb{C}^{p+1}\to \mathbb{C}^{p+1}$ with $i=1,\dots,p+1$, the maps
$$t_i(c_1,\dots,c_{p+1})= (c_1,\dots,c_{i-1},c_i-1,c_{i+1},\dots,c_{p+1}).$$
Then we recall from \cite[Proposition 4.10]{B-ls} that we can write the zero locus of $B_{\tilde{F}}^{\,\mathbf{l}}$ as
\begin{align*}
Z(B_{\tilde{F}}^{\,\mathbf{l}}) &=\bigcup_{i=1}^{p+1}\bigcup_{j=0}^{l_i-1}t_{p+1}^{l_{p+1}}\dots t_{i+1}^{l_{i+1}}t_i^jZ(B_{\tilde{F}}^{\,\mathbf{e}_i})
\end{align*}
where $\mathbf{e}_i$ is the $i$-th standard basis vector.

Now take a small ball $B$  around $(0,\dots,0,-n_W/a_{W,p+1})\in\mathbb{C}^{p+1}$. It follows from the description of $Z(B_{\tilde{F}}^{\,\mathbf{l}})$ above and the definition of ${\mathcal{Z}}$ that only finitely many irreducible components of ${\mathcal{Z}}$ intersect $B$. Denote the reducible variety obtained by taking the union of these components by $Q$. Then $Q$ is an algebraic Zariski closed subset of $\mathbb{C}^{p+1}$.  Since $\mathfrak{Q}\cap B\subset Q$, by taking Zariski closure in $\bC^{p+1}$ we also have $\overline{\mathfrak{Q}\cap B}\subset Q$. Lemma \ref{closureLemma} shows that $\overline{\mathfrak{Q}\cap B}=L$. We thus conclude that
$$L\subset Q\subset {\mathcal{Z}}.$$
We can then find a vector $\mathbf{d}\in \mathfrak{L}$ such that
\begin{align}\label{LinBS}
L\subset Z(B_{\tilde{F}}^{\,\mathbf{d}})=\bigcup_{i=1}^{p+1}\bigcup_{j=0}^{d_i-1}t_{p+1}^{d_{p+1}}\dots t_{i+1}^{d_{i+1}}t_i^jZ(B_{\tilde{F}}^{\,\mathbf{e}_i}).
\end{align}
Suppose that in (\ref{LinBS}), the hyperplane $L$ is contained in $t_{p+1}^jZ(B_{\tilde{F}}^{\,\mathbf{e}_{p+1}})$ for some $j>0$. Apply $t_{p+1}^{-j}$ to the inclusion $L\subset t_{p+1}^jZ(B_{\tilde{F}}^{\,\mathbf{e}_{p+1}})$ to find that
$$\{a_{W,1}s_1+\dots+a_{W,p+1}s_{p+1}+n_W-ja_{W,p+1}=0\}\in Z(B_{\tilde{F}}^{\,\mathbf{e}_{p+1}}).$$
But then $n_W-ja_{W,p+1}<0$, which contradicts the main result of \cite{Gy}. Similarly we find  that for all $i=1,\dots,p$ and $j\ge 0$,
$$L\not\subset t_{p+1}^{d_{p+1}}\dots t_{i+1}^{d_{i+1}}t_i^jZ(B_{\tilde{F}}^{\,\mathbf{e}_i}).$$
We thus conclude that 
$$L\subset Z(B_{\tilde{F}}^{\,\mathbf{e}_{p+1}})$$
However, from the definition of generalized Bernstein-Sato ideals we have $Z(B_{\tilde{F}}^{\,\mathbf{e}_{p+1}})\subset Z(B_{\tilde{F}}),$ so this finishes the proof.
\end{proof}

\begin{lemma}\label{closureLemma}
With the notation of the proof of Proposition \ref{propBSZ} {\color{hot}$(i)$}, $\overline{\mathfrak{Q}\cap B}=L$.
\end{lemma}
\begin{proof}
Fix an arbitrary $\mathbf{l}'=(l_1,\dots,l_p)\in\mathbb{Z}_{>0}^p$. We already remarked that for $l_{p+1}\gg0$, $\mathbf{l}=(l_1,\dots,l_{p+1})\in\mathfrak{L}$. For such sufficiently large $l_{p+1}$ we have moreover that $q_{W,\mathbf{l}}\in \mathfrak{Q}\cap B$. Keeping $\mathbf{l}'$ fixed, but letting $l_{p+1}$ get larger, all the points $q_{W,\mathbf{l}}$ lie on the same line $T_{\mathbf{l}'}$, which has parametric representation
$$T_{\mathbf{l}'} = \left\{\left(0,\dots,0,-\frac{n_W}{a_{W,p+1}}\right)+t\left(l_1,\dots, l_p, -\frac{a_{W,1}l_1+\dots+a_{W,p}l_p}{a_{W,p+1}}\right) \mid t\in\mathbb{C} \right\}.$$
We conclude that all lines of this form are contained in $\overline{\mathfrak{Q}\cap B}$. 
It follows that 
$$L\cap (\mathbb{Q}_{>0}^p\times\mathbb{Q})\subset \overline{\mathfrak{Q}\cap B}.$$
Since the left hand side is clearly Zariski dense inside $L$, we conclude that $L\subset  \overline{\mathfrak{Q}\cap B}$. The other inclusion is immediate since $\mathfrak{Q}\cap B\subset L$.
\end{proof}

\end{subs}

\section{Further remarks}\label{secEx}

\begin{subs}{\bf Jumping numbers.}\label{exGen} Even in the simplest situations, the roots of the $b$-function of the log very-generic polynomials that we produced in this note are not necessarily jumping numbers \cite[Definition 9.3.22]{La}, although small jumping numbers give roots of the $b$-function \cite[9.3.25]{La}.

Let $\mu:Y\ra \bC^2$ be the composition of the blow up at the origin, followed by the blow up of point on the exceptional divisor. Let $g$ a log very-generic polynomial on $\bC^2$, as in Definition \ref{defVGb}, where we fit $\mu$ into the Setup \ref{eqDi} by taking $f=1$ and extending $\mu$ to $\bbm:\bar{Y}\ra\bP^2$ trivially. Let $W_1$ and $W_2$ be the exceptional irreducible divisors of $\mu$, which in this case are the same as those of $\bbm$.  Let
$$
\al_i=\frac{n_{W_i}}{N_{W_i}}=\frac{\ord_{W_i}(K_{\bbm})+1}{\ord_{W_i}g}\quad (i=1, 2).
$$ 
Then both $-\al_1$ and $-\al_2$ are roots of the $b$-function of $g$ by Proposition \ref{propBSZ} {\color{hot} (ii)}. However, either $\al_1$ or $\al_2$ is a jumping number of $g$, but not both at the same time.   

More precisely, let $$
\cA _{rel}=\{(b_1,b_2)\in \bZ^2\mid -(b_1W_1+b_2W_2)\text{ is }\mu\text{-ample}\}
$$ be the integral relatively-ample cone (up to a minus sign). Then $\cA_{rel}$ is the image of the projection of the ample cone $\cA$ of $\bar{Y}$ to the space with coordinates indexed by $W_1$ and $W_2$, cf. \ref{subsAL}. One has
$$\cA_{rel}=\{(b_1,b_2)\in \bN^2\mid 0<b_1<b_2<2b_1\}$$
by applying the numerical relative-ampleness criterion. Let $\cA_{rel}^{vg}$ be the projection to $\cA_{rel}$ of the subcone $\cA^{vg}$ of Definition \ref{defAvg}. Then
$$
\cA^{vg}_{rel} = \cA_{rel}\setminus \left \{(b_1,b_2)\in \cA_{rel}\mid  3b_1=2b_2 \right\}.
$$
To a log very-generic polynomial $g$ one attaches a largely-scaled point $(b_1,b_2)$ in $\cA^{vg}_{rel}$, the coordinates corresponding to the power $L^{\otimes k}$ from which $g$ was generated. Then
$
\al_i=\frac{i+1}{b_i}.
$ Thus one can define for $i=1, 2,$ a subcone
$
\cA_i
$
 of $\cA^{vg}_{rel}$ corresponding to the log very-generic polynomials $g$ having $\al_i$ as jumping number. A short computation reveals that
$$
\cA_1 = \left\{(b_1,b_2) \in \cA^{vg}_{rel} \mid 2b_2<3b_1 \right \} = \cA^{vg}_{rel}\setminus \cA_2,
$$
as claimed.

\begin{rmk}\label{rmkLa} (i) The above explicit description of $\cA^{vg}_{rel}$ as an open subcone of $\cA_{rel}$ also illustrates that in this example ``most" log generic polynomials are  log very-generic polynomials, but the log generic polynomials arising from the integral points of the smaller-dimensional subcone $\{3b_1=2b_2\}$ of $\cA_{rel}$ are not log very-generic polynomials.

(ii) The numbers $k_0$, from the definition of log generic polynomials in the sense of the theorems from the introduction, are not always 1 in this setup. Take $A=2W_1+3W_2$ and $L=\bar\mu^*\cO_{\bP^2}(d)(-A)$. Then $L$ is $\bar\mu$-ample. The intersection $W_1\cap H$ is a set of simple points, away from the point $W_1\cap W_2$, of cardinality $-W_1\cdot A=1$. Hence $\chi(W_1\setminus H\cup W_2)=0$ since $W_1=\bP^1$. In other words, the condition (\ref{eqG1}) fails for $(W_1)_x^\circ$ with $x$ the origin in $\bC^2$ (and $k=1$). Thus one must consider $L^{\otimes k}$ with $k>1$ in this case to obtain log generic polynomials.

\end{rmk}

\end{subs}

\begin{subs} {\bf Generalizations.} The same proofs give the following generalizations:

(i) Theorem \ref{thmMain} and Theorem \ref{thmMain2} hold if $g$ is replaced by a power $g^l$ for any integer $l>0$. 

(ii) Theorem \ref{thmMain} {\color{hot}$(a)$} and Theorem  \ref{thmMain2} {\color{hot}$(a)$} hold simultaneously for all $(f_1^{l_1},\ldots,f_p^{l_p},g)$ with $l_i\in\bN$, respectively  $f^lg$ with $l\in\bN$. That is, the lower bound  on $k\gg 0$ in the definition of $g$ does not depend on $l_i$, respectively on $l$.

(iii) Theorem \ref{thmMain} {\color{hot}$(a)$}, Theorem \ref{thmMain2} {\color{hot}$(a)$}, and Proposition \ref{propCand} hold also with the assumption that $\bar{\mu}$ is a log resolution of the divisor of the rational function $f$ in $\bP^n$ replaced by the weaker assumption that $\mu$ is a log resolution of the polynomial $f$. 

\end{subs}

\begin{subs}\label{subk} {\bf Effective log genericity.} We illustrate the complexity of bounds $k_0$ for the integers $k\ge k_0$ involved in the definition of log (very-)generic polynomials from Theorems \ref{thmMain} and  \ref{thmMain2}.  
Consider the setup and notation as in \ref{subNot1} with the additional conditions from \ref{subNot2}. We assume $n=2$. 
 Define $C(\mu)$ to be maximum number of irreducible components of connected components of the $\mu$-exceptional locus $\cup_{W\in J_{exc}}W$. Next, let $I=J\setminus\{H\}$, the union of $J_{exc}$ with the set of irreducible components of $f^{-1}(0)$. Let $L$ be a very ample line bundle on $\bar{Y}$. Write $L\simeq\bbm^*(\cO_{\bP^2}(d))\otimes\cO_{\bar Y}(-\sum_{W'\in \bar J_{exc}}b_{W'}W')$ for some  integers $b_{W'}, d>0$. Define 
\begin{align*}C(\mu,f, L) =& \max\{ 2, n_W+  a_W, n_Wa_{W''}+a_W, n_Wa_{W'}+n_Wa_W\frac{b_{W'}}{b_W}+a_W \mid\\
& W, W' \in J_{exc}, W''\in I\setminus J_{exc},  W\neq W', W\cap W'\neq \emptyset\neq W\cap W''   \}.
\end{align*} 
One can show that $C(\mu,f,L)\ge C(\mu)$.

\begin{prop}\label{propk0} With the above assumptions,  let $k>0$ be an integer, let $H\in |L^{\otimes k}|$ be generic, and let $g$ be an equation for the curve $\mu(H\cap Y)$ in $\bC^2$. 

(a) If $k>C(\mu)$ then $g$ is log generic in the sense of Theorems \ref{thmMain} and  \ref{thmMain2}, hence the Monodromy Conjecture holds for $\tilde F=(f_1,\ldots,f_p,g)$ and $\tilde f=f_1\ldots f_pg$.

(b) If $k>C(\mu,f,L)$ and $L\in \cA^{vg}$ (see Definition \ref{defAvg}), then $g$ is log very-generic in the sense of Theorems \ref{thmMain} and  \ref{thmMain2}, hence the Strong Monodromy Conjecture holds for $\tilde F=(f_1,\ldots,f_p,g)$ and $\tilde f=f_1\ldots f_pg$.
\end{prop}
\begin{proof}  (a) We run the  proofs from \ref{sub17a} and \ref{subBy}. We check that all claims still hold with the assumption that $k>C(\mu)$ instead of $k\gg 0$. We use the notation from that proof.

Let $x$ be a point  not in the image of any $\mu$-exceptional divisor, but lying on an irreducible component of $\tilde f^{-1}(0)$, say corresponding to the strict transform $W\in J$. Then the germ of $\tilde f_i$ at $x$ is isomorphic to $z^{a_{i,W}}$, $n_W=1$, and it is easy to see that the locus $\{\prod_{i=1}^{p+1}t_i^{a_{i,W}}=1\}$ is contained in $\mathcal P\mathcal Z(Z^{mon}_{\tilde F,x})$. Hence the candidate poles from the strict transforms of the irreducible components of $\tilde f^{-1}(0)$ give monodromy eigenvalues.

The remaining candidate poles come from $W\in J_{exc}$. Note that  $W\simeq\bP^1$ in this case. We assume that $J_{exc}$ is not empty, otherwise there is nothing to prove. 

 Using the numerical criterion for ampleness of $L_{|Y}$ relative  to $\mu:Y\to \bC^2$, we have that the vector $b=(b_{W'})_{W'\in J_{exc}}$ satisfies the condition
\be\label{eqCb}
-\sum_{W'\in J_{exc}}b_{W'} W\cdot W' >0\text{ for all } W\in J_{exc}.
\ee

Take now $x$ a point in the image of some exceptional divisor. Then $J_x= \{W\in J_{exc}\mid W\subset \mu^{-1}(x)\}$ is non-empty and $\cup_{W\in J_x}W$ a connected component of $\cup_{W\in J_{exc}}W$.  Let  $W\in J_x$.  Then, since $H$ is generic,
\be\label{eqApx1}
\chi(W_x^\circ) =\chi(W)-\chi(W\cap H)-\chi(W\cap (\cup_{W'\in I\setminus\{W\}}W'))
\ee
$$
= 2+ k\cdot \sum_{W'\in J_{exc}}b_{W'}W\cdot W' - \sum_{W'\in  I\setminus\{W\}}W\cdot W'.
$$
Thus $\chi(W_x^\circ)\neq 0$ if 
$$k\neq \frac{2- \sum_{W'\in  I\setminus\{W\}}W\cdot W'}{-\sum_{W'\in J_{exc}}b_{W'}W\cdot W'}.
$$
The numerator on the right-hand side is $\le 1$ since there the terms $W\cdot W'$ are $0$ or $1$, and at least one such term is $1$. The denominator is $\ge 1$ by (\ref{eqCb}). Thus the right-hand side is
$\le 1$. Since  $k>C(\mu)\ge 1$ we have that $\chi(W_x^\circ)\neq 0$. Thus (\ref{eqG1}) holds for $k>C(\mu)$.

As in \ref{sub17a}, this guarantees that all the candidate polar hyperplanes from $J_{exc}$ of the topological zeta function of $\tilde F$  appear in the formula (\ref{eqFZ}) for $Z^{mon}_{\tilde F,x}$ before cancellations, for some $x$ as above. Further, we need to guarantee that these candidates do not cancel out of $Z^{mon}_{\tilde F,x}$. As in \ref{sub17a}, we need to guarantee that $\sum_{W\in J'}\chi(W^\circ_x)\neq 0$ for certain subsets $J'$ of $J_{x}$. Repeating the computation from above, it is enough to show  that
$$
k\neq \frac{\sum_{W\in J'}\left(2-\sum_{W'\in I\setminus\{W\}} W\cdot W'\right)}{-\sum_{W\in J'}\sum_{W'\in J_{exc}}b_{W'}W\cdot W'}
$$
for all non-empty $J'\subset J_{x}$. The numerator on the right-hand side is $\le |J'|\le |J_x|\le C(\mu)$, the last inequality due to $\cup_{W\in J_x}W$ being connected.  The denominator is $\ge 1$. Thus the right-hand side is $\le C(\mu)<k$, so $\sum_{W\in J'}\chi(W^\circ_x)\neq 0$. Thus (\ref{eqG2}) holds for $k>C(\mu)$. This finishes the proof of part (a).

(b) We run the  proofs from  \ref{subSMC}. We check that all claims still hold with the assumption that $k>C(\mu,f,L)$ instead of $k\gg 0$. The condition that the divisor in (\ref{eqCx}) is nef and big is satisfied easily using the bound $C(\mu,f,L)\ge 2$ since $\cO_W(K_W)=\cO_{\bP^1}(-2)$. Now we only need to show that (\ref{eqSMX}) holds for $k>C(\mu,f,L)$. So, let $W\in J_{exc}$. Since $n=2$, there are only three types of elements of $\bar J_W$:
$$
\bar J_W = (J_W\cap J_{exc})\cup  (J_W\cap (I\setminus J_{exc})) \cup \{H\}.
$$
We show first that (\ref{eqSMX}) holds for $W'\in J_W\cap J_{exc}$, that is, for $W'\in J_{exc}$ with $W\neq W'$ and $W\cap W'\neq\emptyset$. Using (\ref{eqbW}), we see that we must show that
$$
k\neq \frac{n_Wa_{W'}-ta_W}{b_Wt-n_Wb_{W'}}\quad\text{ for all }t\in \bZ.
$$
Note that the denominator is not zero since $L$ is in $\cA^{vg}$. If the right hand side is the constant function 0 or 1 in $t$, the claim holds since $k>C(\mu,f,L)>1$. If $a_W=0$ but $a_{W'}\neq 0$ the claim also holds, since then the right-hand side is still $\le C(\mu,f,L)$. In the remaining cases, the right-hand side represents the values on $\bZ$ of a real function $h(t)=(A-Bt)/(Ct-D)$ with $A, B, C, D\in \bZ_{\ge 0}$, $B,C,D\neq 0$,  and $D/C\not \in \bZ$. This set of values is maximized by the value at the integer closest to $D/C$ in the direction of $A/B$, that is by $\lfloor D/C\rfloor$ or $\lceil D/C\rceil$ depending on the sign of $AC-BD$. In any case, $h(t)\le A +Bt$ for positive $t$, hence the maximal value for $t\in\bZ$ of the right-hand side is 
$$\le A+(D/C +1)B=n_Wa_{W'}+n_Wa_W\frac{b_{W'}}{b_W}+a_W\le C(\mu,f,L)<k.$$
This settles this case.

For the other two cases for $W'$, the claim is easier to prove. If $W'=H$ then one uses the bound $C(\mu,f,L)\ge n_W+a_W$. If $W'\in J_W\cap(I\setminus J_{exc})$, one uses the bound  $C(\mu,f,L)\ge n_Wa_{W'}+a_W$.
\end{proof}

\begin{rmk}
(i) Note that $C(\mu)$ does not depend on $L$ nor on the orders of vanishing of $f$ along components in $J$. The same is true for $C(\mu,1,L)$, but for general $f$ we do not know how to find a lower bound for $k$ independent of $L$ guaranteeing log very-genericity for $g$ as in  Proposition \ref{propk0}.

(ii) If $n\ge 3$, we do not know how to provide a lower bound for $k$ independent of $L$ guaranteeing log genericity. The problem arises already with the analog of (\ref{eqApx1}); if $n=3$ this becomes a quadratic polynomial and the natural upper bounds on its roots depend on $L$.

\end{rmk}

\end{subs}

\begin{subs}\label{subnondeg} {\bf Non-degenerate polynomials.} We recall the definitions of Newton polyhedron and non-degeneracy.  Let $g\in\bC[x_1,\ldots, x_n]$ with $g(0)=0$, and let $g=\sum_{m\in \bN^n}c_mx^m$ with $c_m\in\bC$ be the unique monomial decomposition of $g$. The {\it Newton polyhedron of $g$} (at the origin) is the convex hull $\Gamma(g)$ in $\bR^n_{\ge 0}$ of $\cup_{m\in\text{supp}(g)}(m+\bR^n_{\ge 0})$, where $\text{supp}(g)=\{m\in\bN^n\mid c_m\neq 0\}$. The polynomial $g$ is {\it non-degenerate} (at the origin)  if for any compact face $\tau$ of $\Gamma(g)$, the polynomials ${\partial f_\tau}/{\partial x_i}$, $1\le i\le n$, do not have a common zero in   $(\bC^*)^n=\{x_1x_2\ldots   x_n\neq 0\}$, where $f_\tau=\sum_{m\in\tau\cap\bN^n}c_mx^m$.

Let $\tau$ be a  face of a Newton polyhedron $\Gamma$. Let $\bR^S$ be the minimal coordinate subspace of $\bR^n$ containing $\tau$. Let $\text{aff}(\tau)$ be the affine span of $\tau$ in $\bR^S$. Let $s_\tau=|S|$ and assume $\dim \tau=s_\tau-1$. Then $\text{aff}(\tau)$ is given by an  equation $a_1e_1+\ldots+a_se_s=N$ with $e_i$ forming the standard basis of $\bR^S$, $a_i, N\in\bN$, and $\text{gcd}(a_1,\ldots, a_s, N)=1$. Set $N(\tau)=N$ and $\nu(\tau)=a_1+\ldots+a_s$.

\begin{proof}[Proof of Theorem \ref{thmnondeg}.] We can assume $g$ is  a singular germ. Since $g$ is non-degenerate, all the poles different than $-1$ of the local topological zeta function $Z^{top}_{g,0}(s)$ of $g$ at the origin are included in the set
$$
\left\{-\frac{\nu(\tau)}{N(\tau)}\mid \tau\text{ is a a facet of }\Gamma(g)\text{ not lying in a coordinate hyperplane}\right\},
$$
by a result of Denef \cite{D} for the $p$-adic local zeta function, but the statement for the local topological zeta function follows similarly, see \cite[Th\'eor\`eme 5.3 (ii)]{DLn}. Our assumption on $\Gamma$ implies that all the facets of $\Gamma(g)=k\Gamma$ not lying in a coordinate hyperplane are compact.

It is enough to show that every candidate pole as above gives a zero or a pole of the monodromy zeta function $Z^{mon}_{g,x}(t)$ of $g$ at some point $x$ close to the origin in $\bC^n$.

Fix $\tau_0$ a compact facet of $\Gamma(g)$ not lying in a coordinate hyperplane.  Since $g$ is non-degenerate, by Varchenko \cite{V} we have
$$
Z^{mon}_{g,0}(t)=\prod_\tau (t^{N(\tau)}-1)^{(-1)^{\dim \tau}\cdot(\dim \tau)!\cdot\text{Vol}_\bZ(\tau)}
$$
where the product runs over compact faces $\tau$ of $\Gamma(g)$ with $\dim\tau=s_\tau-1$. Here, if $\dim \tau =0$, one sets $\text{Vol}_\bZ(\tau)=1$. For the other compact faces $\tau$, the latticial volume $\text{Vol}_\bZ(\tau)$ of $\tau$ is computed by declaring that on $\text{aff}(\tau)$ the cube spanned by the lattice basis of $\bZ^n\cap\text{aff}(\tau)$ has volume 1. 

Since $\tau_0$ is compact, $\lambda=e^{-2\pi i\nu(\tau_0)/N(\tau_0)}$ is a root or a pole  of the term in $Z^{mon}_{g,0}(t)$ corresponding to $\tau_0$. We show that $\lambda$ is a root or a pole of  $Z^{mon}_{g,0}(t)$.

Since $\Gamma(g)=k\Gamma$, the faces $\tau$ of $\Gamma(g)$ are into 1-1 correspondence with the faces $\tau'$ of $\Gamma$. Under this correspondence $\tau=k\tau'$. Moreover,  $$(-1)^{\dim \tau}(\dim \tau)!\text{Vol}_\bZ(\tau)=
 (-1)^{\dim \tau'}(\dim \tau')!  \text{Vol}_\bZ(\tau') k^{\dim \tau' } 
$$
for $k>0$. If $\lambda$ is not a root or a pole of $Z^{mon}_{g,0}(t)$, then 
$$
0= (-1)^{n-1}(n-1)!k^{n-1} \sum_{\tau'} \text{Vol}_\bZ(\tau') + (\text{lower order terms in }k)
$$
for $k\gg 0$, where the sum runs over all compact facets $\tau'$ of $\Gamma$ such that $\lambda^{N(k\tau')}=1$. We have that $\sum_{\tau'} \text{Vol}_\bZ(\tau')>0$ since $\tau'_0$ contributes non-trivially to the sum. This is a contradiction for $k\gg 0$, hence $\lambda$ is a root or a pole of  $Z^{mon}_{g,0}(t)$.
\end{proof}

\end{subs}

\end{document}